\documentclass{amsart}
\usepackage[top=1 in,bottom=1.25 in,left=2cm, right=2cm]{geometry}
\usepackage{amsmath, amsthm, amscd, amsfonts, amssymb, color}
\usepackage[dvips]{graphicx}
\usepackage{tikz}
\usepackage{float}
\usepackage{perpage}
\usepackage[T1]{fontenc}
\newtheorem{thm}{Theorem}[section]

\newtheorem{lem}[thm]{Lemma}
\newtheorem{prop}[thm]{Proposition}
\theoremstyle{definition}

\numberwithin{equation}{section}

\begin{document}

\title[$\ell^1$-Cospectrality of graphs]{$\ell^1$-Cospectrality of graphs}%
\author[A. Abdollahi]{Alireza Abdollahi}%
\address{Department of Mathematics, University of Isfahan, Isfahan 81746-73441, Iran; and School of Mathematics, Institute for Research in Fundamental Sciences (IPM), P.O. Box 19395-5746, Tehran, Iran}%
\email{a.abdollahi@math.ui.ac.ir}%
\author[N. Zakeri]{Niloufar Zakeri}
\address{Department of Mathematics, University of Isfahan, Isfahan 81746-73441, Iran}%
\email{zakeri@sci.ui.ac.ir}%
\thanks{}%

\begin{abstract}
The following problem  has been proposed in  [Research problems from the Aveiro workshop on graph spectra, {\em Linear Algebra and its Applications}, {\bf 423} (2007)
172-181.]:\\
(Problem AWGS.4) Let $G_n$ and $G'_n$ be two nonisomorphic graphs on $n$ vertices with spectra
$$\lambda_1 \geq \lambda_2 \geq \cdots \geq \lambda_n \;\;\;\text{and}\;\;\; \lambda'_1 \geq \lambda'_2 \geq \cdots \geq \lambda'_n,$$
respectively. Define the distance between the spectra of $G_n$ and $G'_n$ as
$$\lambda(G_n,G'_n) =\sum_{i=1}^n (\lambda_i-\lambda'_i)^2 \;\;\; \big(\text{or use}\; \sum_{i=1}^n|\lambda_i-\lambda'_i|\big).$$
Define the cospectrality of $G_n$ by
$$\text{cs}(G_n) = \min\{\lambda(G_n,G'_n) \;:\; G'_n \;\;\text{not isomorphic to} \; G_n\}.$$
{\bf Problem A.} Investigate $\text{cs}(G_n)$ for special classes of graphs.\\

In this paper we study Problem A for certain graphs with respect to the $\ell^1$-norm, i.e. $\sigma(G_n,G'_n)=\sum_{i=1}^n|\lambda_i-\lambda'_i|$. We find $\text{cs}(K_n)$, $\text{cs}(nK_1)$, $\text{cs}(K_2+(n-2)K_1)$ ($n\geq 2$), $\text{cs}(K_{n,n})$ and $\text{cs}(K_{n,n+1})$, where $K_n, nK_1, K_2+(n-2)K_1, K_{n,m} $ denote the complete graph on $n$ vertices,  the null graph on $n$ vertices, the disjoint union of the $K_2$ with $n-2$ isolated vertices  ($n\geq 2$), and the complete bipartite graph with parts of sizes $n$ and $m$, respectively.
\end{abstract}
\maketitle
\section{\bf Introduction and Results}
Throughout the paper all graphs are simple, that is finite and undirected without loops and multiple edges.\\
Richard Brualdi proposed in  \cite{s} the following problem:\\
(Problem AWGS.4) Let $G_n$ and $G'_n$ be two nonisomorphic graphs on $n$ vertices with spectra
$$\lambda_1 \geq \lambda_2 \geq \cdots \geq \lambda_n \;\;\;\text{and}\;\;\; \lambda'_1 \geq \lambda'_2 \geq \cdots \geq \lambda'_n,$$
respectively. Define the distance between the spectra of $G_n$ and $G'_n$ as
$$\lambda(G_n,G'_n) =\sum_{i=1}^n (\lambda_i-\lambda'_i)^2 \;\;\; \big(\text{or use}\; \sum_{i=1}^n|\lambda_i-\lambda'_i|\big).$$
Define the cospectrality of $G_n$ by
$$\text{cs}(G_n) = \min\{\lambda(G_n,G'_n) \;:\; G'_n \;\;\text{not isomorphic to} \; G_n\}.$$
Let $$\text{cs}_n = \max\{\text{cs}(G_n) \;:\; G_n \;\;\text{a graph on}\; n \;\text{vertices}\}.$$
This function measures how far apart the spectrum of a graph with $n$ vertices can be from the
spectrum of any other graph with $n$ vertices.\\

\noindent {\bf Problem A.} Investigate $\text{cs}(G_n)$ for special classes of graphs.\\
{\bf Problem B.} Find a good upper bound on $\text{cs}_n$.\\

See, for example \cite{clss,wz,zp} some applications of spectral distances of graphs. 
In \cite{ajo}, Problem B has completely been answered. 
It is of course possible to study Problems A or B for other matrix representations of graphs, such as Laplacian, normalized Laplacian, signless Laplacian and distance matrices, see e.g., \cite{ahk,lld,ds,ghl,ha,j,jzs}. In the current paper, we only study Problem A with respect to the adjacency matrix of graphs.

In \cite{ao,o}, Problem A is studied and  cospectralities of  classes of complete graphs and complete bipartite  graphs with respect to Euclidean norm (the $\ell^2$-norm) are computed.

In \cite{js}, spectral distance between certain graphs is studied with respect to the $\ell^1$-norm i.e. $\sigma(G_n,G'_n)=\sum_{i=1}^n|\lambda_i-\lambda'_i|$. 

In this paper we study Problem A for some graphs with respect to the $\ell^1$-norm. \\

Let us first introduce some notations.\\

For a graph $G$, $V(G)$ and $E(G)$ denote the vertex set and edge set of $G$, respectively; By the  order and size of $G$ we mean the number of vertices and the number of edges of $G$, respectively; Denote by $\overline{G}$ the complement of $G$. Let $G$ be a graph with vertex set $\{v_1,\dots,v_n\}$. The adjacency matrix of $G$ is an $n\times n$ matrix $A(G)=[a_{ij}]$ such that $a_{ij}=1$ if $v_i$ and $v_j$ are adjacent, and $a_{ij}=0$ otherwise. By the eigenvalues of $G$, we mean those of its adjacency matrix. We denote by $Spec(G)$ the multiset of the eigenvalues of the graph $G$. For two graphs $G$ and $H$ with disjoint vertex sets, $G+H$ denotes the graph with the vertex set $V(G) \cup V(H)$ and the edge set $E(G) \cup E(H)$, i.e. the disjoint union of two graphs $G$ and $H$. The complete product (join) $G \nabla H$ of graphs $G$ and $H$ is the graph obtained from $G+H$ by joining every vertex of $G$ with every vertex of $H$. In particular, $nG$ denotes $\underbrace{G+\cdots+G}_n$ and $\nabla_nG$ denotes $\underbrace{G\nabla G \nabla \dots \nabla G}_n$.

For positive integers $n_1, \dots,n_\ell$, $K_{n_1,\dots,n_\ell}$ denotes the complete multipartite graph with $\ell$ parts of sizes $n_1,\dots,n_\ell$.
Let $K_n$ denote the complete graph on $n$ vertices, $nK_1=\overline{K_n}$ denote the null graph on $n$ vertices and $P_n$ denote the path with $n$ vertices.\\

   
Here we find the cospectralities with respect to the $\ell^1$-norm of those graphs  which have been already found their ones with respect to the $\ell^2$-norm in \cite{ao}. In the following, we give the table of cospectralities of these graphs with respect to the $\ell^1$-norm and $\ell^2$-norm. In the third column (fifth column, resp.)  all graphs $H$ with $\text{cs}(G)=\sigma(G,H)$ ($\text{cs}(G)=\lambda(G,H)$, resp.) are given for the graph $G$ in the first column.  \\

\begin{small}
\begin{tabular}{|c|c|c|c|c|}\hline
Graph & $\ell^1$-norm & $H$ & $\ell^2$-norm & $H$\\ \hline
$nK_1$& 2& $K_2+(n-2)K_1$& 2 & $K_2+(n-2)K_1$\\
$K_2+(n-2)K_1$& $2(\sqrt{2}-1)$&$P_3+(n-3)K_1$&$2(\sqrt{2}-1)^2$&$P_3+(n-3)K_1$\\
$K_n$&2&$K_n\setminus e$ or $K_{n-1}+K_1$ &$n^2+n-n\sqrt{n^2+2n-7}-2$& $K_n\setminus e$\\
$K_{n,n}$&$ 2(n-\sqrt{n^2-1})$&$K_{n-1,n+1}$&$2(n-\sqrt{n^2-1})^2$&$K_{n-1,n+1}$\\
$K_{n,n+1}$&$2(\sqrt{n^2+n}-\sqrt{n^2+n-2})$&$K_{n-1,n+2}$&$2(\sqrt{n^2+n}-\sqrt{n^2+n-2})^2$&$K_{n-1,n+2}$\\\hline
\end{tabular}\\
\end{small}

 Our main results are as follows.

\begin{thm}\label{empty}
 For every integer $n\geq2$, $\text{cs}(nK_1)=2$. Moreover, $\text{cs}(nK_1)=\sigma(nK_1,H)$ for some graph $H$ if and only if $H\cong K_2+(n-2)K_1$.
\end{thm}

\begin{thm}\label{oneedge}
$\text{cs}(K_2)=\sigma(K_2,2K_1)=2$ and for every integer $n\geq3$, $\text{cs}(K_2+(n-2)K_1)=2(\sqrt{2}-1)$. Moreover,  $\text{cs}(K_2+(n-2)K_1)=\sigma(K_2+(n-2)K_1,H)$ for some graph $H$ if and only if $H\cong P_3+(n-3)K_1$.
\end{thm}

\begin{thm}\label{kn}
Let $n\geq 2$ be an integer. Then $\text{cs}(K_n)=2$. Moreover, $\text{cs}(K_n)=\sigma(K_n,H)$ for some graph $H$ if and only if $H\cong K_{n-1}+K_1$ or $H\cong K_n\setminus e$ for any edge $e$, where $K_n\setminus e$ is the graph obtaining from $K_n$ by deletion one edge $e$.
\end{thm}

\begin{thm}\label{knnn}
For every integer $n\geq 2$, $\text{cs}(K_{n,n})=2(n-\sqrt{n^2-1})$. Moreover, $\text{cs}(K_{n,n})=\sigma(K_{n,n},H)$ for some graph $H$ if and only if
$H\cong K_{n-1,n+1}$.
\end{thm}

\begin{thm}\label{knn+1}
For every integer $n\geq 2$, $\text{cs}(K_{n,n+1})=2(\sqrt{n^2+n}-\sqrt{n^2+n-2})$. Moreover, $\text{cs}(K_{n,n+1})=\sigma(K_{n,n+1},H)$ for some graph $H$ if and only if
$H\cong K_{n-1,n+2}$.
\end{thm}

\section{\bf $\ell^1$-Cospectrality  of graphs with at most one edge}
In this section we will determine the cospectrality  of
 the graphs with at most one edge. Let $G$ be a simple graph of order $n$ and $\lambda_1\geq\cdots\geq \lambda_n $ be the eigenvalues of $G$. Recall that the energy of $G$ is defined as $E(G)=\displaystyle\sum_{i=1}^{n} |\lambda_i|$. 
We need the following Theorem in the sequel.\\

\begin{thm}[See \cite{ccgh}]\label{eng} 
Let $G$ be a graph with m edges. Then\\
 $$E(G)\geq 2\sqrt{m},$$ \\
with equality if and only if $G$ is a complete bipartite graph plus arbitrarily isolated vertices.
\end{thm}
 
\noindent{\bf Proof of Theorem \ref{empty}.} Let $H$ be a simple graph of order $n$ and size $m$. Clearly, we have $\sigma(nK_1,H)=E(H)$. Since $H$ is not isomorphic to $nK_1$, by Theorem \ref{eng}, the minimum value of $E(H)$ is 2 and it happens whenever $m=1$. This shows that $H\cong K_2+(n-2)K_1$. $\hfill \Box$ \\

\noindent{\bf Proof of Theorem \ref{oneedge}.} It is easy to see that $\text{cs}(K_2)=\sigma(K_2,2K_1)=2$. Suppose $H$ is a simple graph of order $n$ and size $m$ with eigenvalues $\lambda_1\geq \cdots \geq \lambda_n$. For every integer $n\geq 3$, 
it follows from Theorem \ref{eng} and triangular inequality that 
\begin{eqnarray*}
\sigma(K_2+(n-2)K_1,H)&=&\mid \lambda_1 - 1\mid + \displaystyle\sum_{i=2}^{n-1} \mid \lambda_i \mid + \mid \lambda_n+1 \mid \\
 &=& E(H)-\mid\lambda_1\mid-\mid\lambda_n\mid+\mid\lambda_1-1\mid+\mid\lambda_n+1\mid \\
 &\geq & 2\sqrt{m}-\mid\lambda_1\mid-\mid\lambda_n\mid+|\lambda_1|-1+|\lambda_n|-1\\
 &\geq& 2(\sqrt{m}-1),
\end{eqnarray*}
Note that if $m=0$, then $\sigma(K_2+(n-2)K_1,H)=2$. Now assume that $m>0$. 
Since $H$ is not isomorphic to $K_2+(n-2)K_1$, the minimum value of $\sigma(K_2+(n-2)K_1,H)$ is $2(\sqrt{2}-1)$ and it happens whenever $m=2$. The graphs with two edges are $2K_2+(n-4)K_1$ and $P_3+(n-3)K_1$ and  
$$\sigma(K_2+(n-2)K_1, 2K_2+(n-4)K_1)=2,$$
$$\sigma(K_2+(n-2)K_1, P_3+(n-3)K_1)=2(\sqrt{2}-1).$$
It shows that $H\cong P_3+(n-3)K_1$. This completes the proof. $\hfill \Box$
 
\section{\bf $\ell^1$-Cospectrality of the complete graph}

In this section we show that for every integer $n\geq 2$, the minimum value of $\sigma(K_n,H)$ happens whenever $H\cong K_n\setminus e$ or $K_{n-1}+K_1$, and $\text{cs}(K_n)=2$. We need the following results.

\begin{thm}[\cite{js}, part (i) of Theorem 3.4]\label{jov}
Let $G$ be an arbitrary graph, and let $n^*$ denote the number of its eigenvalues which are greater than or equal to $-1$. Then the following holds:
$$\sigma(K_n,G)=2\big(n^{*}-1+\displaystyle\sum_{i=2}^{n^{*}}\lambda_i(G)\big).$$
\end{thm}

\begin{lem} [See \cite{ao}]\label{speckne}
For every integer $n\geq2$ and every arbitrary edge $e$ of $K_n$,
$$Spec(K_n\setminus e)=\Big\{\frac{n-3+\sqrt{n^2+2n-7}}{2},0,\underbrace{-1,\ldots,-1}_{n-3},\frac{n-3-\sqrt{n^2+2n-7}}{2}\Big\}.$$
\end{lem}

\begin{lem} [See \cite{ao}]\label{lemS}
Let $G$ be a graph with eigenvalues $\lambda_1\geq \lambda_2\geq \lambda_3\geq \cdots \geq \lambda_n$.
 The graph $G$ is isomorphic to one the following graphs if and only if $\lambda_1>0$, $\lambda_2\leq 0$ and $\lambda_3<0$.
\begin{enumerate}
\item $G\cong K_n$,
\item $G\cong K_1+K_{n-1}$, \item $G\cong K_{2,1,\dots,1}=
K_n\setminus e$ for an edge $e$ of $K_n$.
\end{enumerate}
\end{lem}

\noindent{\bf Proof of Theorem \ref{kn}.} The result easily  follows for $n=2$. So we may assume that $n\geq 3$. It is easy to see that $\sigma(K_n,K_{n-1}+K_1)=2$. By Theorem \ref{jov}, Lemma \ref{speckne} and the fact the sum of eigenvalues of a simple graph is zero, we obtain
\begin{eqnarray*}
\sigma(K_n,K_n\setminus e)&=&2\big(n-2+\displaystyle\sum_{i=2}^{n-1}\lambda_i(K_n\setminus e)\big)\\
&=&2 \big(n-2-\lambda_1(K_n\setminus e)-\lambda_n(K_n\setminus e)\big)\\
&=&2.
\end{eqnarray*}
Let $H$ be a simple graph of order $n$ and let $n^*$ denote the number of eigenvalues of $H$ which are greater than or equal to $-1$. To complete the proof, it is sufficient to show that if $\sigma(K_n,H)\leq2$ then $H$ is isomorphic to one of the following graphs: $K_n$, $K_{n-1}+K_1$ or $K_n\setminus e$. By Theorem \ref{jov},
$$\sigma(K_n,H)=2\big(n^*-1+\displaystyle\sum_{i=2}^{n^*}\lambda_i(H)\big).$$
If $\sigma(K_n,H)\leq 2$, it follows that 
\begin{equation}\label{eq1}
n^*+\displaystyle\sum_{i=2}^{n^*}\lambda_i(H)\leq 2.
\end{equation}
Let ${n^*}_{\geq 0}$ and ${n^*}_{<0}$ denote the cardinality of the following sets $A$ and $B$, respectively,
$$A=\{ \lambda_i(H) \mid 0\leq \lambda_i(H) \},$$
$$B=\{ \lambda_i(H) \mid -1\leq \lambda_i(H)<0 \}.$$
By Inequality \ref{eq1}, 
$${n^*}_{\geq 0}+{n^*}_{<0}+\displaystyle\sum_{\lambda_i(H)\in A \; \text{and}\; i>1}\lambda_i(H)+\displaystyle\sum_{\lambda_i(H)\in B}\lambda_i(H)\leq 2.$$
Since ${-n^*}_{<0}\leq\displaystyle\sum_{\lambda_i(H)\in B}\lambda_i(H)$, 
\begin{equation}\label{eq2}
{n^*}_{\geq 0}+\displaystyle\sum_{\lambda_i(H)\in A}\lambda_i(H)\leq 2.
\end{equation}
Since $\sigma(K_n,nK_1)=2n-2$ and $n\geq 3$, we can assume that $H$ has at least one edge. It implies that $\lambda_1(H)> 0$ and so ${n^*}_{\geq 0}\geq 1$. By Inequality \ref{eq2}, we have the following cases:\\

\noindent {\bf Case 1.} \; ${n^*}_{\geq 0}=1$. Then both of $\lambda_2(H)$ and $\lambda_3(H)$ are negative.\\

\noindent {\bf Case 2.} \; ${n^*}_{\geq 0}=2$. Then $\lambda_2(H)=0$ and $\lambda_3(H)<0$.\\

Therefore by Lemma \ref{lemS}, $H$ is isomorphic to one of the following graphs: $K_n$, $K_{n-1}+K_1$ or $K_n\setminus e$. This completes the proof. $\hfill \Box$

\section{\bf $\ell^1$-Cospectrality of complete bipartite graphs}
We need the following results to prove  Theorem \ref{knnn}.

\begin{thm}[Theorem 9.1.1 of \cite{Godsil}]\label{interlacing} Let $G$ be a graph of order $n$ and $H$ be an induced subgraph of $G$ with order $m$. Suppose that $\lambda_1(G)\geq\cdots\geq\lambda_n(G)$ and $\lambda_1(H)\geq\cdots\geq\lambda_m(H)$ are the eigenvalues of $G$ and $H$, respectively. Then for every $i$, $1\leq i\leq m$, $\lambda_i(G)\geq\lambda_i(H)\geq\lambda_{n-m+i}(G)$.
\end{thm}

\begin{thm}[Theorem 1 of \cite{CY}] \label{eigenvalue}
Let $G$ be a simple graph of order $n$ without isolated vertices. If $\lambda_2(G)$ is the second largest eigenvalue of $G$, then
\begin{enumerate}
\item $\lambda_2(G)=-1$ iff $G$ is a complete graph with at least two vertices.
\item $\lambda_2(G)=0$ iff $G$ is a complete $k$-partite graph with $2\leq k\leq n - 1$.
\item There exists no graph $G$ such that $-1 \leq \lambda_2(G) \leq 0$.
\end{enumerate}
\end{thm}

\begin{thm}[See \cite{S}, and also Theorem 6.7 of \cite{CDS}] \label{Sm} A graph has exactly one positive eigenvalue if and only if its non-isolated vertices form a complete multipartite graph.
\end{thm}

\begin{thm}[Theorem 2 of \cite{CY}] \label{1/3} Let $G$ be a graph of order $n$ without isolated vertices. Then $0<\lambda_2(G)<\frac{1}{3}$
if and only if $G\cong (K_1+K_2)\nabla\overline{K_{n-3}}$, where
$\lambda_2(G)$ is the second largest eigenvalue of $G$.
\end{thm}

\begin{prop}[Proposition 4.1 of \cite{ao}]\label{kmn}
Let $m$ and $n$ be positive integers. Then $\text{cs}(K_{m,n})>0$ if and only if  the minimum of $x+y$ for all positive integers  $x,y$ such that $xy=mn$
is  attained on $\{x,y\}=\{m,n\}$.
\end{prop}

By Proposition \ref{kmn}, if $n>0$, then $\text{cs}(K_{n,n})>0$. Now we compute  $\text{cs}(K_{n,n})$.\\ 
The method that we use below for the proof is similar to that of Theorem 1.4 of \cite{ao}.\\

\noindent{\bf Proof of Theorem \ref{knnn}.}
Since for all positive integers $p$ and $q$ 
$$ Spec(K_{p,q})=\{\sqrt{pq},\underbrace{0,\ldots,0}_{p+q-2},-\sqrt{pq}\},$$  
it implies that $\sigma(K_{n,n},K_{n-1,n+1})=2(n-\sqrt{n^2-1})$.
By direct computing, the result follows for $n=2$ and $n=3$. Now we can assume that $n\geq 4$. Suppose, for a contradiction, that there exists a graph $H$ of order $2n$ such that $H$ is not isomorphic to either $K_{n,n}$ or  $K_{n-1,n+1}$, and also $\sigma(K_{n,n},H)\leq 2(n-\sqrt{n^2-1})$. Let $\lambda_1\geq\cdots\geq\lambda_{2n}$ be the eigenvalues of $H$.
 Since for $n\geq 4$, $2(n-\sqrt{n^2-1})<\frac{1}{3}$, it follows that $|\lambda_2|<\frac{1}{3}$.
By Theorem \ref{eigenvalue}, we have $0\leq\lambda_2<\frac{1}{3}$. Now it remains to investigate the following cases:\\

\noindent {\bf Case 1.} \; Assume that $\lambda_2=0$. If $\lambda_1=0$, then $H\cong \overline{K_{2n}}$ and so
$\sigma(K_{n,n},H)=2n\geq8$, a contradiction. Hence we can suppose that $\lambda_1>0$. By Theorem \ref{Sm},
there exist some positive integers $k$ and $n_1,\dots,n_k$ and an
integer $t\geq0$ such that $H\cong \overline{K_t}+K_{n_1,\ldots,n_k}$. If
$k=1$, then $H\cong \overline{K_t}+K_{2n-t}$ such that $1\leq t\leq 2n-2$. We have
$$\sigma(K_{n,n},H)=\mid n-t-1\mid+ 3n-t-3.$$
One can to see that $\sigma(K_{n,n},H)>2(n-\sqrt{n^2-1})$, a contradiction. If $k=2$, then
$H\cong \overline{K_t}+K_{p,q}$ for some $p$ and $q$ such that $p+q=2n-t$.
In this case we have $\sigma(K_{n,n},H)=2\mid n-\sqrt{pq}\mid$. It is
not difficult to see that if $\{p,q\}\not=\{n,n\}$ and
$\{p,q\}\not=\{n-1,n+1\}$, then
$2\mid n-\sqrt{pq}\mid >2(n-\sqrt{n^2-1})$. Therefore $k\geq 3$. If
$n_1=\cdots=n_k=1$, then $H\cong\overline{K_t} + K_{2n-t}$ and so
$\sigma(K_{n,n},H)> 2(n-\sqrt{n^2-1})$, a contradiction. Now we
may assume that $n_i\geq 2$, for some $1\leq i \leq k$. So $H$ has $K_{1,1,2}$ as an induced subgraph. Since $Spec(K_{1,1,2})=\{2.56155, 0, -1, -1.56155\}$ and
$\lambda_3(K_{1,1,2})=-1$, by 
Theorem \ref{interlacing}, we have $\mid\lambda_{2n-1}\mid\geq 1$ and so
$\sigma(K_{n,n},H)\geq 1$, a contradiction. \\

\noindent {\bf Case 2.} \; Suppose that $0<\lambda_2<\frac{1}{3}$. By Theorem \ref{1/3}, there exists an integer $t\geq0$ such that $H\cong
\overline{K_t}+(K_1+K_2)\nabla\overline{K_{2n-t-3}}$. Let $2n-t-3=1$. Since $Spec((K_1+K_2)\nabla K_1)=\{2.17009, .31111, -1, -1.48119\}$, it is not hard to see that $\sigma(K_{n,n},H)>2(n-\sqrt{n^2-1})$, a
contradiction. So we can assume that $2n-t-3>1$. Then $H$ has $K_{1,1,2}$ as an induced subgraph and the rest is similar to the previous part.\\
This completes the proof. $\hfill\Box$ \\

We need the following results to proof Theorem \ref{knn+1}.

\begin{thm}[See {\cite{p}}]\label{petro}  
Let $G$ be a graph without isolated vertices and let $\lambda_2(G)$ be the second largest eigenvalue of $G$. Then $0<\lambda_2(G)\leq\sqrt{2}-1$ if and only if one of the following holds:
\begin{enumerate}
\item $G\cong (\nabla_{t}(K_1+K_2))\nabla K_{n_1,\ldots,n_m},$
\item $G\cong (K_1+K_{r,s})\nabla\overline{K_q},$
\item $G\cong (K_1+K_{r,s})\nabla K_{p,q}.$
\end{enumerate}
\end{thm}
Now we prove an ``$\ell^1$-version'' of  Lemma 2.7 of \cite{o}.
\begin{lem}\label{knm}
Let $m$ and $n$ be two positive integers and $G$ be a graph of order $n + m$. Suppose that there are no positive integers $r, s$ and a non-negative integer $t$ such that
$G\cong K_{r,s}+tK_1$. If $\lambda_2(G)\leq \sqrt{2}-1$, then $\sigma(G,K_{m,n})\geq 1$.
\end{lem}
\begin{proof}
By Theorem \ref{eigenvalue}, we can consider the following cases:\\

\noindent {\bf Case 1.} \; $\lambda_2(G)=-1$. Thus $G\cong K_{n+m}$. It is not hard to see that $\sigma(G,K_{m,n})\geq 1$.\\

\noindent {\bf Case 2.} \; $\lambda_2(G)=0$. If $\lambda_1(G)=0$, then $G\cong \overline{K_{m+n}}$. Therefore 
 $\sigma(G,K_{m,n})=2\sqrt{mn}\geq 2$. Now we suppose that $\lambda_1(G) > 0$. By Theorem \ref{Sm}, there are some positive integers $k, n_1,\ldots, n_k$ and a non-negative integer t such that $G\cong K_{n_1,\ldots,n_k}+tK_1$. If $k=1$, then $G\cong K_{n+m-t}+tK_1$. We have 
$$\sigma(G,K_{m,n})=|m+n-t-1-\sqrt{mn}|+(m+n-t-2)+|1-\sqrt{mn}|,$$
Since $\lambda_2(G)=0$ and $\lambda_1(G)> 0$, $1\leq t\leq m+n-2$ and $m,n\geq 2$. So $\sigma(G,K_{m,n})\geq 1$. If $k=2$, then $G\cong K_{n_1,n_2}+tK_1$, a contradiction. Let $k\geq 3$. If $n_1=\cdots=n_k=1$, then $G\cong K_{n+m-t}+tK_1$. Therefore $\sigma(G,K_{m,n})\geq 1$. Now we can assume that $n_i\geq 2$, for some $1\leq i \leq k$. So $G$ has $K_{1,1,2}$ as an induced subgraph. Since $\lambda_3(K_{1,1,2})=-1$, by Theorem \ref{interlacing}, we have $\mid\lambda_{n+m-1}(G)\mid\geq 1$ and so $\sigma(G,K_{m,n})\geq 1$. \\

\noindent {\bf Case 3.} \; $0<\lambda_2(G)\leq \sqrt{2}-1$. By Theorem \ref{petro}, $G$ is isomorphic to one the following graphs: $(\nabla_h(K_1+K_2))\nabla K_{n_1,\ldots,n_k}+tK_1$, $(K_1+K_{r,s})\nabla\overline{K_q}+tK_1$ or $(K_1+K_{r,s})\nabla K_{p,q}+tK_1$, where $k, r, s, p, q$ and $n_1,\dots,n_k$ are some positive integers
and $t, h$ are two non-negative integers. Let $G\cong (\nabla_h(K_1+K_2))\nabla K_{n_1,\ldots,n_k}$. If $h=0$, then $G\cong K_{n_1,\ldots,n_k}$, that is considered in the previous part. If $h\geq 1$, then $(K_1+K_2)\nabla K_1$ is an induced sungraph of $G$. Since $\lambda_3((K_1+K_2)\nabla K_1)=-1$, by Theorem \ref{interlacing}, $\mid\lambda_{n+m-1}\mid\geq 1$ and so $\sigma(G,K_{m,n})\geq 1$. Also $(K_1+K_2)\nabla K_1$ is an induced subgraph for other cases and the result follows by Theorem \ref{interlacing}.
\end{proof}
\noindent{\bf Proof of Theorem \ref{knn+1}.}
We have $\sigma(K_{n,n+1},K_{n-1,n+2})=2(\sqrt{n^2+n}-\sqrt{n^2+n-2})$. 
Suppose, for a contradiction, that $H$ is a graph not isomorphic to both of $K_{n,n+1}$, $K_{n-1,n+2}$ and $\text{cs}(K_{n,n+1})=\sigma(K_{n,n+1},H)$. By direct computing of cospectralities of all graphs of orders at most $10$, one finds that the order of $H$ is at least $11$. We may assume that $n=5$. Note that $\sigma(K_{5,6},K_{4,7})=2(\sqrt{30}-\sqrt{28})$. 
If there are positive integers $r$, $s$ and a non-negative integer $t$ such that $H\cong K_{r,s}+tK_1$ and $r+s+t = 11$, then $\sigma(K_{5,6},H)=2|\sqrt{30}-\sqrt{rs}|> 2(\sqrt{30}-\sqrt{28})$.  So we can assume that there are no positive integers $r$, $s$ and a non-negative integer $t$ such that $H\cong K_{r,s}+tK_1$ and $r+s+t= 11$. Since $\sigma(K_{5,6},H)\leq 2(\sqrt{30}-\sqrt{28})<\sqrt{2}-1$, $|\lambda_2(H)|<\sqrt{2}-1$. By Lemma \ref{kmn}, $\sigma(K_{5,6},H)\geq 1$, a contradiction. We conclude that the result holds for $n=5$. Now we may assume that $n\geq 6$ and by similar arguments given in Theorem \ref{knnn}, the result follows. $\hfill\Box$ 

\section*{{\bf Acknowledgments}}

The research of the first author was in part supported by a grant from School of Mathematics, Institute for Research in Fundamental Sciences(IPM).


\end{document}